\documentclass[letterpaper]{article}
\usepackage{proceed2e}
\usepackage[margin=1in]{geometry}
\usepackage{enumerate}
\usepackage{amsmath,amscd,amssymb,amsthm}
\usepackage{graphicx}
\usepackage{url}

\usepackage{times}

\theoremstyle{plain}
   \newtheorem{theorem}{Theorem}[section]
   
   \newtheorem{lemma}[theorem]{Lemma}
   \newtheorem{corollary}[theorem]{Corollary}
   
   \newtheorem{problem}[theorem]{Problem}
\theoremstyle{definition}

\newtheorem{example}{Example}[section] 

\theoremstyle{remark}
 \newtheorem{remark}{Remark}[section]

\newcommand{\ZZ}{\mathbb{Z}}
\newcommand{\R}{\mathbb{R}}
\newcommand{\s}{\mathcal{S}}
\newcommand{\m}{{\it m}}

\def\newop#1{\expandafter\def\csname #1\endcsname{\mathop{\rm
#1}\nolimits}}

\newop{sort}
\newop{init}
\newop{diag}
\newop{supp}
\newop{mult}
\newop{Sym}
\newop{si}
\newop{glb}
\newop{des}
\newop{fmaj}
\newop{comaj}
\newop{maj}
\newop{lhp}

\title{Counting Markov Equivalence Classes by Number of Immoralities}

\author{} 

\author{ {\bf Adityanarayanan Radhakrishnan} \\Laboratory for Information and \\Decision
Systems, and Institute for \\Data, Systems and Society\\Massachusetts Institute of Technology\\
Cambridge, MA, USA\\
\And
{\bf Liam Solus}  \\
Department of Mathematics\\
KTH Royal Institute of Technology\\
	Stockholm, Sweden \\
\And
{\bf Caroline Uhler}   \\Laboratory for Information and \\Decision
Systems, and Institute for \\Data, Systems and Society\\
Massachusetts Institute of Technology\\
Cambridge, MA, USA\\
}

\begin{document}

\maketitle

\begin{abstract}
Two directed acyclic graphs (DAGs) are called Markov equivalent if and only if they have the same underlying undirected graph (i.e.~skeleton) and the same set of immoralities.
Using observational data, a DAG model can only be determined up to Markov equivalence, and so it is desirable to understand the size and number of Markov equivalence classes (MECs) combinatorially.  
In this paper, we address this enumerative question using a pair of generating functions that encode the number and size of MECs on a skeleton $G$, and in doing so we connect this problem to classical problems in combinatorial optimization.  
The first is a graph polynomial that counts the number of MECs on $G$ by their number of immoralities.      
Using connections to the independent set problem, we show that computing a DAG on $G$ with the maximum possible number of immoralities is NP-hard.  
The second generating function counts the MECs on $G$ according to their size.  
Via computer enumeration, we show that this generating function is distinct for every connected graph on $p$ nodes for all $p\leq 10$.  
\end{abstract}

\section{INTRODUCTION}
\label{sec: introduction}
Graphical models based on directed acyclic graphs (DAGs) are widely used to represent complex causal systems in applications ranging from computational biology to epidemiology, and sociology (Friedman et al., 2000; Pearl, 2000; Robins et al., 2000; Spirtes et al., 2001).  
A DAG entails a set of conditional independence (CI) relations through the Markov properties.  
Two DAGs are said to be \emph{Markov equivalent} if they entail the same CI relations.  
Verma and Pearl (1992) show that a \emph{Markov equivalence class} (MEC) is determined by the underlying undirected graph (or \emph{skeleton}) and the placement of immoralities, i.e.~induced subgraphs of the form $X\to Z\leftarrow Y$.
From observational data alone, a DAG can only be identified up to Markov equivalence, and it is therefore important to describe the set of MECs and their sizes.  
For instance, if the MECs are large in size, then causal inference algorithms that operate in the space of MECs as compared to DAGs could significantly increase efficiency.  
This paper focuses on the complexity of deciding the number and size of MECs on a fixed skeleton.    

The literature on the MEC enumeration problem is surprisingly sparse, and it can be summarized in terms of two important perspectives: (1) enumerate all MECs on $p$ nodes (as in (Gillespie and Perlman, 2001)), and (2) enumerate all MECs of a given size (as in (Gillespie, 2006; Steinsky, 2003; Wagner, 2013; He and Yu, 2016)).  
The characterization of Markov equivalence of Verma and Pearl  (1992) results in a representation of a MEC by a graph with directed and undirected edges known as the \emph{essential graph} (Andersson et al., 1997) (or \emph{cPDAG} (Chickering, 2002) or \emph{maximally oriented graph} (Meek, 1995)).
Gillespie and Perlman (Gillespie et al., 2001) use this characterization to identify all MECs on $p\leq 10$ nodes;  
namely, they fix a skeleton on $p$ nodes, and then count the number of ways to compatibly place immoralities within the skeleton.  
The study of He and Yu (2016) counts the size of a \emph{single} MEC in terms of the \emph{core graph} of its essential graph.  
The works (Gillespie, 2006; Steinsky, 2003; Wagner 2013) give inclusion-exclusion formulae for MECs of a fixed size by utilizing the structure of the essential graph.  
However, since essential graphs can be complicated, these formulae are only realizable for constrained classes of MECs.
Namely, (Steinsky, 2003) and (Wagner, 2013) only consider MECs of size one, and (Gillespie, 2006) must fix the undirected edges of the essential graphs to be enumerated.  
These results show that the enumeration of MECs by number of nodes and size is a difficult problem in general.  

A common approach to difficult graphical structure enumeration problems is to specify a type of graph for which to solve the problem.  
The purpose of this paper is to study the MEC enumeration problem from this perspective.  
This approach is used in classic combinatorial optimization problems such as the enumeration of independent sets, matchings, and colorings (Ellis-Monaghan and Merino, 2011; Levit and Mandrescu, 2005).  
In each case, it is typical to consider a generating function $P(G;x) := \sum_{k\geq0}\alpha_k(G)x^k$, that when evaluated at $x=1$ returns the desired statistic for the graph $G$.  
Such a generating function is called a \emph{graph polynomial}.  
For example, the \emph{independence polynomial} of $G$ is the graph polynomial in which $\alpha_k(G)$ is the number of independent sets of size $k$ in $G$.  
Researchers can gain useful information about the original enumeration problem by studying the properties and coefficients of a graph polynomial.  
For instance, the independence polynomial of $G$ encodes the total number of independent sets (or the \emph{Fibonacci number} of $G$) (Prodinger and Tichy, 1982), the maximum size of an independent set (or \emph{independence number} of $G$) (van Lint and Wilson, 2001), and the number of independent sets of a fixed size (Levit and Mandrescu, 2005); all of which have been studied extensively.   
These refined statistics work together to give a complete understanding of the problem of enumerating independent sets for $G$.  
Here, we are interested in the combinatorial statistics of a graph $G$:
\begin{enumerate}[(1)]
	\item $M(G)$, the total number of MECs on $G$, 
	\item $\m(G)$, the maximum number of immoralities on $G$, 
	\item $\m_k(G)$, the number of ways to place exactly $k$ immoralities on $G$, and 
	\item $M(G)_{\mbox{freq}} := (s_1(G),s_2(G),\ldots)$, where $s_k(G)$ denotes the number of MECs on $G$ of size $k$.  
\end{enumerate}
The first three statistics fit together naturally in the graph polynomial presentation
$$
M(G;x) := \sum_{k=0}^{\m(G)}\m_k(G)x^k,
$$
since then 
$ 
M(G;1) = M(G).
$
Similarly, we can express the vector $M(G)_{\mbox{freq}}$ as a graph polynomial of the form $\sum_{k\geq0}s_k(G)x^k$.  
However, it is perhaps more natural to encode the entries of $M(G)_{\mbox{freq}}$ in the arithmetic function
$$
S(G;x) := \sum_{k\geq0}\frac{s_k(G)}{k^x}.  
$$
We then have that $S(G;0)=M(G)=M(G;1)$, and both generating functions are multiplicative with respect to disjoint unions of graphs.  
That is,  
\begin{equation*}
\begin{split}
M(G\sqcup H; X)&=M(G;x)M(H;x), \mbox{ and}	\\
S(G\sqcup H; X)&=S(G;x)S(H;x),
\end{split}
\end{equation*}
where $G\sqcup H$ denotes the disjoint union of two graphs $G$ and $H$.  

We are interested in the complexity of computing the combinatorial statistics associated to the generating functions $M(G;x)$ and $S(G;x)$.  
The resulting theorems inform the complexity of computing the number and size of MECs on a fixed skeleton $G$.  
Our first main theorem, proven in Section~\ref{subsec: proof of theorem np-complete}, observes the difficulty of computing a MEC on a skeleton $G$ with many immoralities.  
\begin{theorem}
\label{thm: np-complete}
Given an undirected graph $G$, the problem of computing a DAG $\mathcal{G}$ with skeleton $G$ and $m(G)$ immoralities is NP-hard.  
\end{theorem}
In analogy to the independence number of $G$, we call $\m(G)$ the \emph{immorality number of $G$}.  
This number is natural to consider when one attempts to enumerate the MECs on $G$ by counting all compatible placements of immoralities.  
Here, we use the notion of NP-hardness as defined in (Chapter 5, Garey and Johnson, 1979).  
We will prove Theorem~\ref{thm: np-complete} in Section~\ref{sec: immorality numbers and star decompositions} via a reduction of the minimum vertex cover problem.  
To do so, we first prove a correspondence between minimum vertex covers of a given triangle-free graph $G$ and minimum decompositions of $G$ into non-overlapping stars, which we call \emph{minimum star decompositions}.  
As with most NP-hard problems, restricting to special cases can make the problem tractable.  
In this case, the connection with minimum star decompositions allows us to compute $\m(G)$ in some special cases.  
In particular, we can compute $\m(G)$ for triangle-free graphs whose minimum star decompositions are isomorphic as forests, and we apply this result to recover $\m(G)$ for the complete bipartite graph $K_{p,p}$ and some special types of circulant graphs.  

Our second complexity result is computer-aided and observational.  
In order to study the generating functions $M(G;x)$ and $S(G;x)$, we developed a computer program for the enumeration of the combinatorial statistics (1), (2), (3), and (4) that expands on the original program of Gillespie and Perlman (2001).  
Using this program, we prove the following intriguing theorem.  
\begin{theorem}
\label{thm: frequency determinism}
The arithmetic generating function $S(G;x)$ is distinct for every connected graph $G$ on $p$ nodes for all $p\leq 10$.   
\end{theorem}
Theorem~\ref{thm: frequency determinism} can be viewed as a complexity result in the sense that it tells us that recovering the complete set of statistics $M(G)_{\mbox{freq}}$ for some unknown connected graph $G$ from observational data alone is equally hard as recovering $G$ itself.  
In Section~\ref{sec: computational analysis}, we describe our computer program for the computation of the statistics (1), (2) (3), and (4), and verify Theorem~\ref{thm: frequency determinism}.

\section{IMMORALITY NUMBERS AND STAR DECOMPOSITIONS}
\label{sec: immorality numbers and star decompositions}
In this section, we show that computing the immorality number $\m(G)$ of a graph $G$ is an NP-hard problem by showing that it is a reduction of the problem of computing minimum vertex covers of $G$.  
Recall that a \emph{vertex cover} of $G$ is a subset $S$ of vertices of $G$ for which each edge of $G$ is adjacent to some vertex in $S$.  
A classic problem in combinatorial optimization is to identify a vertex cover of minimum size for a given graph $G$.  
Formally stated, this is the search problem
\begin{problem}
\label{prob: minimum vertex cover}
{\bf MINIMUM VERTEX COVER}

\noindent\underline{INPUT}: An undirected graph $G= (V,E)$.  

\noindent\underline{OUTPUT}: A subset $C\subset V$ such that for all edges $\{u,v\}\in E$ either $u\in C$ or $v\in C$ and $|C|$ is minimized with respect to this property.
\end{problem}
The decision version of this problem is called VERTEX COVER (Karp, 1972) and is stated as follows.
\begin{problem}
\label{prob: vertex cover}
{\bf VERTEX COVER}

\noindent\underline{INPUT}: An undirected graph $G= (V,E)$ and a nonnegative integer $k$.  

\noindent\underline{PROPERTY}: $G$ has a vertex cover of size less than or equal to $k$.  
\end{problem}

A search problem $\Pi$ is said to be \emph{NP-hard} if there is a polynomial time Turing reduction from an NP-complete $\Pi^\prime$ problem to $\Pi$ (Chapter 5, Garey and Johnson, 1979). 
That is, if we are given a polynomial time algorithm $A$ for solving $\Pi$, then there exists a polynomial time algorithm for solving $\Pi^\prime$ using $A$ as a hypothetical subroutine.
In (Poljak, 1974), it is shown that VERTEX COVER is NP-complete even when restricted to triangle-free graphs.  
Moreover, given a polynomial time algorithm for solving MINIMUM VERTEX COVER, we can certainly derive a polynomial time algorithm to solve VERTEX COVER (for triangle-free graphs or otherwise).  
Thus, MINIMUM VERTEX COVER is NP-hard for both triangle-free and arbitrary graphs.  
Analogously, we consider the following search and decision problems related to the computation of the immorality number $m(G)$.  
\begin{problem}
\label{prob: maximum immoralities}
{\bf MAXIMUM IMMORALITIES}

\noindent\underline{INPUT}: An undirected graph $G= (V,E)$.  

\noindent\underline{OUTPUT}: A DAG $\mathcal{G}$ with skeleton $G$ and $m(G)$ immoralities.
\end{problem}
\begin{problem}
\label{prob: immoralities}
{\bf IMMORALITES}

\noindent\underline{INPUT}: An undirected graph $G= (V,E)$ and a nonnegative integer $k$.  

\noindent\underline{PROPERTY}: There is a DAG $\mathcal{G}$ with skeleton $G$ having at least $k$ immoralities. 
\end{problem}

In the following, we will identify a polynomial time Turing reduction of MINIMUM VERTEX COVER  to MAXIMUM IMMORALITIES when restricted to triangle-free graphs. 
A polynomial time solution to MAXIMUM IMMORALITIES would trivially yield a polynomial time solution to the same problem in the triangle-free case.  
Since this would in turn solve an NP-complete problem, we can conclude that the general instance of MAXIMUM IMMORALITIES is NP-hard.  
This will prove Theorem~\ref{thm: np-complete}.  

In order to reduce MINIMUM VERTEX COVER to MAXIMUM IMMORALITIES for triangle-free graphs, we will utilize a notion of \emph{star decompositions} of $G$.  
We then use this connection to compute $\m(G)$ in the special cases of the complete bipartite graph $K_{p,p}$ and a family of circulant graphs.   

\subsection{STAR DECOMPOSITIONS}
\label{subsec: star decompositions}
Let $G = (V,E)$ be a connected, undirected graph, and let $K_{p,q}$ denote the complete bipartite graph with nodes partitioned into a collection of size $p$ and a collection of size $q$.  
Recall that a \emph{$p$-star} is the complete bipartite graph $K_{1,p}$ and its \emph{center} is the unique degree $p$ node.  
A collection of stars $\{S_1,\ldots,S_k\}$ is called a \emph{star decomposition of $G$} if each $S_i$ is a subgraph of $G$ and each edge of $G$ is an edge of exactly one star in the collection.  
Our definition of star decomposition is a bit more general than the standard notion studied in graph decompositions.  
The classic notion of a star decomposition adds the requirement that the stars $S_1,\ldots,S_k$ are all isomorphic to one another.  
While the literature on which graphs admit a star decomposition of this type is quite extensive (Cain, 1974; Cohen and Tarsi, 1991; Ushio et al., 1978), there is substantially less work relating to the more general notion we use here (Lin and Shyu, 1996).  

In the following, the \emph{trivial star} refers to $K_{1,0}$, and the \emph{size} of a star $S$ is the size of its edge set, which we denote by $|S|$.  
The \emph{size} of a star decomposition $\s$ is the number of stars in the decomposition, and it is denoted $|\s|$.  
Given a star decomposition $\s=\{S_1,\ldots,S_k\}$ let $v(\s)\in\R^k$ denote the vector of the sizes of stars in $\s$ ordered greatest-to-least from left-to-right.  
So if $|S_1|\geq\cdots\geq|S_k|$ then $v(\s) = (|S_1|,|S_2|,\ldots,|S_k|)$.  
If $\s$ is a star decomposition of size $k$ with cardinality vector $v(\s)\in\R^k$, for $m\geq k$ we embed $v(\s)\in\R^m$ by appending zeros to the right end of $v(\s)\in\R^k$.  
Notice that this corresponds to appending trivial stars to $\s$.  
We call a star decomposition of $G$ \emph{reduced} if it contains no trivial stars.
Notice that the largest reduced star decomposition contains at most $|E|$ stars.  
A \emph{minimum} star decomposition of $G$ contains the minimum number of stars over all star decompositions of $G$.  
Notice that a minimum star decomposition will always be reduced.
Since the maximum number of stars in a reduced star decomposition of $G$ is $|E|$, then any minimum star decomposition contains at most $|E|$ stars.  
Also, given a star decomposition $\s$, we call the set of all centers of stars in $\s$ the \emph{center set of $\s$}, and we denote it by $C(\s)$.  
Note that if a star consists only of a single edge, then we simply choose one of its endpoints to be the center node.

For any DAG $\mathcal{G}$ on the undirected graph $G$ we can construct a star decomposition of $G$ as follows.  
For each node $v\in V$, consider the substar $S_v$ in $G$ whose center is $v$ and whose edges are those directed into $v$ in the DAG $\mathcal{G}$.  
The \emph{star decomposition of $G$ induced by $\mathcal{G}$} is then
$$
\s(\mathcal{G}) := \{S_v : v\in V\}.
$$
Notice that an induced star decomposition will not be reduced, and may contain intervals $K_{1,1}$.  
\begin{remark}
\label{remark: not all star decompositions are induced}
Not all star decompositions of a graph $G$ are induced by some DAG on $G$.   
For example, any graph has a star decomposition consisting of precisely its set of edges.  
In the case of the $4$-cycle, for instance, this decomposition cannot arise from a DAG.  
\end{remark}
Since a star decomposition induced by a DAG always contains at least one trivial star, we make the following important definition.  
A minimum star decomposition of $G$ is \emph{induced} by a DAG $\mathcal{G}$ on $G$ if it is a reduction of the star decomposition induced by $\mathcal{G}$.

\begin{example}
\label{ex: minimum star decompositions}
Consider $C_4$, the cycle on $4$ nodes.  
Up to isomorphism, $C_4$ admits the three reduced star decompositions depicted in Figure~\ref{fig: 4-cycle star decompositions}.  
From this we can see that the minimum star decompositions of $C_4$ are all isomorphic to $\{K_{1,2},K_{1,2}\}$.  
The two right-most star decompositions in Figure~\ref{fig: 4-cycle star decompositions} are each induced by DAGs.  
For example, the middle decomposition is induced by the DAG $\mathcal{G}_1$ and the right-most decomposition is induced by the DAG $\mathcal{G}_2$ as depicted in Figure~\ref{fig: 4-cycle induced star decompositions}. 
The left-most star decomposition in Figure~\ref{fig: 4-cycle star decompositions} is the maximum cardinality reduced star decomposition of the $4$-cycle, which consists of exactly one copy of $K_{1,1}$ for each edge of $C_4$.  
\end{example}
	\begin{figure}
	\centering
	\includegraphics[width=0.45\textwidth]{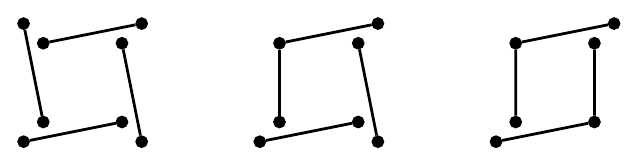}
	\caption{The three nonisomorphic reduced star decompositions of the $4$-cycle}
	\label{fig: 4-cycle star decompositions}
	\end{figure}
	\begin{figure}[b!]
	\centering
	\includegraphics[width=0.45\textwidth]{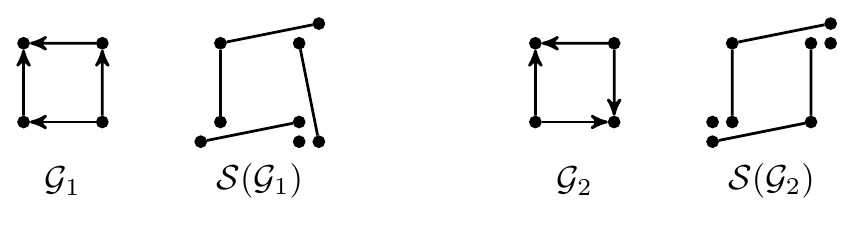}
	\caption{The DAGs $\mathcal{G}_1$ and $\mathcal{G}_2$ and their induced (nonreduced) star decompositions.}
	\label{fig: 4-cycle induced star decompositions}
	\end{figure}

Example~\ref{ex: minimum star decompositions} demonstrates the properties of minimum star decompositions that we will use to study the immorality number of triangle-free graphs.  
Notice first that the center set of each star decomposition is a vertex cover of $C_4$ and that the minimum vertex covers of $G$ are center sets of minimum star decompositions.  
Indeed, there exists a many-to-one correspondence between minimum star decompositions and minimum vertex covers of $G$.

\begin{lemma}
\label{lem: center set is a minimum vertex cover}
Suppose $\s$ is a minimum star decomposition with center set $C(\s)$.  
Then $C(\s)$ is a minimum vertex cover of $G$.  
\end{lemma}

\begin{proof}
Recall that the center set $C(\s)$ of any star decomposition $\s$ is a vertex cover of $G$.  
Therefore, any minimum star decomposition has to be at least as large as any minimum vertex cover of $G$.  
Suppose that for any minimum vertex cover $C$ of $G$ we can find a star decomposition of $G$ with center set $C$.  
Then it follows that any minimum star decomposition has size exactly that of a minimum vertex cover of $G$.  
Moreover, the center set of any minimum star decomposition must be a minimum vertex cover.  
Thus, to complete the proof, we need only show that any minimum vertex cover of $G$ is the center set of some star decomposition of $G$.   

For a node $v$ of a graph $G$ we let $N[v]$ denote the neighbors of $v$ in $G$ including the node $v$ itself.   
Suppose that $C = \{c_1,\ldots,c_k\}$ is a minimum vertex cover of $G$.  
Let $\s(C) = \{S_1,\ldots,S_k\}$ denote the star decomposition of $G$ given by setting
\begin{equation*}
\begin{split}
S_1 &:= \langle N[c_1] \rangle, \\
S_i &:= \langle N[c_i]\backslash\left(\cup_{j\leq i}N[c_j]\right)\rangle, \mbox{for $i>1$}.\\
\end{split}
\end{equation*}
Since $\s$ is a star decomposition of $G$ with center set $C$, this completes the proof.
\end{proof}

\begin{lemma}
\label{lem: minimum vertex cover gives star decompositions}
Suppose $C$ is a minimum vertex cover of $G$ and $\s$ is any star decomposition of $G$ with center set $C$.  
Then $\s$ is a minimum star decomposition.   
\end{lemma}

\begin{proof}
Recall that the center set $C(\s)$ of any star decomposition $\s$ of $G$ is a vertex cover of $G$.  
Thus, just as stated in the proof of Lemma~\ref{lem: center set is a minimum vertex cover}, we know that any star decomposition of $G$ is at least as large as any minimum vertex cover of $G$.  
By the construction in the proof of Lemma~\ref{lem: center set is a minimum vertex cover}, we know in fact that this lower bound is tight.  
Thus, any star decomposition with center set that is a minimum vertex cover must have minimum size.
\end{proof}

\begin{lemma}
\label{lem: minimum star decompositions of maximum weight}
Suppose $\s$ is a minimum star decomposition of $G$ with cardinality vector $v(\s)\in\R^{|E|}$ such that $v(\s)^Tv(\s)$ is maximum over all star decompositions of $G$.  
Then $\s$ is induced by some DAG with skeleton $G$.  
\end{lemma}

\begin{proof}
Note first that any star decomposition $\s = \{S_1,\ldots,S_k\}$ of $G$ is induced by some directed, but not necessarily acyclic, graph $\mathcal{G}(\s)$.  
Namely, $\mathcal{G}(\s)$ is the directed graph whose arrows are given by directing all edges of $S_i$ so that their heads are at the center node of $S_i$ for all $i\in[k]$.  
Since each edge of $G$ appears in exactly one star in $\s$, this definition yields a unique directed graph. 

For the sake of contradiction, suppose $\s$ is a minimum star decomposition of $G$ for which $v(\s)^Tv(\s)$ is maximized, but $\s$ is not induced by a DAG.  
Then $\s$ is induced by the directed graph $\mathcal{G}(\s)_0:=\mathcal{G}(\s)$ constructed in the previous paragraph.  
By assumption, $\mathcal{G}(\s)_0$ contains some directed cycles.  
Notice that if $v$ is any node contained in a directed cycle, then by the construction of $\mathcal{G}(\s)_0$ we know $v\in C(\s)$, since $v$ has nonzero indegree in $\mathcal{G}(\s)_0$.  
Therefore, all vertices in all directed cycles in $\mathcal{G}(\s)_0$ lie in the center set $C(\s)$.  

Let $v_0$ be a node of highest $G$-degree that is contained in a directed cycle in $\mathcal{G}(\s)_0$. Reorient all arrows of $\mathcal{G}(\s)_0$ so that $v_0$ is a sink, and denote the resulting directed graph by $\mathcal{G}(\s)_1$.  Notice that the directed cycles in $\mathcal{G}(\s)_1$ are precisely the directed cycles of $\mathcal{G}(\s)_0$ that do not use the node $v_0$.  
In particular, $\mathcal{G}(\s)_1$ contains strictly less directed cycles than $\mathcal{G}(\s)_0$.  
This is because changing node $v_0$ into a sink eliminated precisely the directed cycles that passed through $v_0$.  

We iterate this procedure as follows: Let 
$$v_i\in V\setminus \bigcup_{j=1}^{i-1} N[v_j]$$
be a node of highest $G$-degree that is contained in a directed cycle in $\mathcal{G}(\s)_i$. Reorient all arrows of $\mathcal{G}(\s)_i$ so that $v_i$ is a sink, and denote the resulting directed graph by $\mathcal{G}(\s)_{i+1}$. Note that $\mathcal{G}(\s)_{i+1}$ contains strictly less directed cycles than $\mathcal{G}(\s)_i$. Therefore, iterating this process must result in some DAG $\mathcal{G}(\s)_m$. 

The center set of the (reduced) induced star decomposition of $\mathcal{G}(\s)_m$ is contained in the center set of $\s$, since all nodes on any directed cycle in $\mathcal{G}(\s)_0$ are contained in $C(\s)$.  
Therefore, since $\s$ is a minimum star decomposition, so is $\s(\mathcal{G}(\s)_m)$. 

Now consider the associated vectors $v(\s), v(\s(\mathcal{G}(\s)_m))\in\R^{|E|}$.  
Since at each step in the construction of $\mathcal{G}(\s)_m$ we selected a center node $v_i$ of highest possible degree, then $v(\s)\prec_{\mbox{lex}}v(\s(\mathcal{G}(\s)_m))$, and in particular 
$$
v(\s)^Tv(\s)<v(\s(\mathcal{G}(\s)_m))^Tv(\s(\mathcal{G}(\s)_m)), 
$$
which is a contradiction.  
\end{proof}

In some special instances, when the minimum star decompositions of a graph $G$ are well-understood, we can use this theory to compute $\m(G)$.  
Recalling Example~\ref{ex: minimum star decompositions}, notice that the minimum star decompositions of $C_4$ are all isomorphic to one another as forests, and each minimum star decomposition of $C_4$ is induced by a DAG.  
With this example in mind, we prove the following theorem. 
\begin{theorem}
\label{thm: minimum star decompositions formula}
Let $G$ be a triangle-free, undirected graph whose minimum star decompositions are all isomorphic to one another as forests.  
Then given any minimum star decomposition
$
\s(G) = \{S_1,\ldots,S_k\}
$
of $G$ the immorality number of $G$ is
$$
\m(G) = \sum_{i=1}^k{|S_i|\choose 2}.
$$
\end{theorem}

\begin{proof}
Since the maximum size of a minimum star decomposition is $|E|$, we can simply assume $k=|E|$ by filling out the set with trivial stars.  
That is, without loss of generality we assume that all star decompositions considered have the same cardinality $k = |E|$, but may contain trivial stars.  
A minimum star decomposition is then simply one with the maximum number of trivial stars.  

Recall that for every DAG $\mathcal{G}$ on $G$ we can construct the induced star decomposition $\s(\mathcal{G}) = \{S_1,\ldots,S_k\}$.  
Since $G$ is triangle-free, the number of immoralities in $\mathcal{G}$ is precisely 
$
\sum_{i=1}^k{|S_i|\choose 2}.
$
Each such induced star decomposition admits a vector in $\R^k$ for each permutation $\sigma\in S_k$ of cardinalities
$
(|S_{\sigma(1)}|,|S_{\sigma(2)}|,\ldots,|S_{\sigma(k)}|)\in\R^k,
$
and we let $v(\mathcal{G})$ denote any one of these vectors.  
More generally, any star decomposition $\s$ of $G$ admits such a vector of cardinalities for each permutation $\sigma\in S_k$, any one of which we denote by $v(\s)\in\R^k$.  
Let $\mathcal{V}(G)$ denote the set of all possible choices of vectors $v(\s)$ for all possible star decompositions of $G$.  
Then our goal is to maximize the objective function
$
\sum_{i=1}^k{x_i\choose 2}
$
over the set $\mathcal{V}(G)\subset\ZZ^k_{\geq0}$.  
Since the objective function satisfies
$$
\sum_{i=1}^k{x_i\choose 2} = \frac{1}{2}\left(\sum_{i=1}^kx_i^2-\sum_{i = 1}^kx_i\right),
$$
and for all $(x_1,\ldots,x_k)\in\mathcal{V}(G)$, we have that
$
\sum_{i = 1}^kx_i = |E| = k,
$
then we are interested in solving the integer optimization problem
\begin{center}
\begin{tabular}{r c c l}
maximize	&	&	& ${\bf x}^T{\bf x}$				\\
subject to	&	&	& $\sum_{i=1}^kx_i = k$,			\\
		&	&	& ${\bf x}\in\ZZ^k_{\geq 0}$,	\\
		&	&	& ${\bf x}\in\mathcal{V}(G)$.				\\
\end{tabular}
\end{center}
The presentation of this optimization problem is redundant, but it is to emphasize the fact that any vector in $\mathcal{V}(G)$ lies in the $k^{th}$ dilate of the probability simplex $\Delta_k$, which we denote by $k\Delta_k$.  
Therefore, we are simply maximizing the length over all vectors in the probability simplex that also lie in the set $\mathcal{V}(G)$.  
Since the value of ${\bf x}^T{\bf x}$ strictly increases as we approach the boundary of $k\Delta_k$ then the star decompositions with the maximum number of trivial stars will yield the maximum value of the objective function.  These are the minimum star decompositions, all of which are isomorphic as trees, and therefore have the same vectors $v(\s)$ up to a permutation of coordinates.  
Since we have assumed that at least one of these star decompositions is induced by a DAG, it follows that the maximum value of the original objective function 
$
\sum_{i=1}^k{x_i\choose 2}
$ 
is the immorality number of $G$.  
\end{proof}

Collectively, Lemmas~\ref{lem: center set is a minimum vertex cover}, \ref{lem: minimum star decompositions of maximum weight}, and Theorem~\ref{thm: minimum star decompositions formula} allow us to prove Theorem~\ref{thm: np-complete}.
\subsection{PROOF OF THEOREM~\ref{thm: np-complete}}
\label{subsec: proof of theorem np-complete}
Let $G$ be a triangle-free graph, and suppose that we have a polynomial time algorithm that returns a DAG $\mathcal{G}^*$ with skeleton $G$ for which $\mathcal{G}^*$ has the maximum number of immoralities.  
By Lemma~\ref{lem: minimum star decompositions of maximum weight} and Theorem~\ref{thm: minimum star decompositions formula}, we know that the maximum value of 
$
\sum_{i = 1}^{|E|}{|S_i|\choose 2}
$
is achieved by a minimum star decomposition induced by a DAG.  
Since the value of 
$
\sum_{i = 1}^{|E|}{|S_i|\choose 2}
$
 is exactly equal to the number of immoralities in a DAG with a triangle-free skeleton, it follows that our DAG $\mathcal{G}^*$ induces a minimum star decomposition $\s(\mathcal{G}^*)$ that maximizes 
 $
\sum_{i = 1}^{|E|}{|S_i|\choose 2}.
$
We know by Lemma~\ref{lem: center set is a minimum vertex cover} that the center set $C(\s(\mathcal{G}^*))$ is a minimum vertex cover of $G$.  
Therefore, we have a polynomial time algorithm for computing a minimum vertex cover of the triangle-free graph $G$. 
It is clear that a polynomial time algorithm for MAXIMUM IMMORALITIES for arbitrary graphs trivially yields a polynomial time algorithm for MAXIMUM IMMORALITIES for triangle-free graphs.  
Therefore, since MINIMUM VERTEX COVER is NP-complete for triangle-free graphs, we know that the general instance of MAXIMUM IMMORALITIES is NP-hard.  
This completes the proof of Theorem~\ref{thm: np-complete}.   \hfill$\square$

\begin{remark}
\label{rmk: self-reducibility}
Recall that there is trivially a polynomial time Turing reduction of MINIMUM VERTEX COVER to VERTEX COVER.  
Conversely, it is well-known that VERTEX COVER is \emph{self-reducible}.  
That is, given a polynomial time algorithm for VERTEX COVER one can find a polynomial time algorithm solving MINIMUM VERTEX COVER.  
Collectively, this says that solving MINIMUM VERTEX COVER is no more or no less hard than solving VERTEX COVER.  
Since the former direction is trivial, the critical observation made here is the self-reducibility of VERTEX COVER.  

The proof of self-reducibility for VERTEX COVER is standard across many NP-complete structural search problems for graphs, and it goes as follows.  
Given a graph $G = (V,E)$, the minimum size of a vertex cover must be between $0$ and $|V|$.  
Thus, by a binary search, we can determine in polynomial time the size $k^\ast$ of a minimum vertex cover of $G$.  
Then, to recover a vertex cover $C$ with size $k^\ast$ of $G$, we first pick a vertex $v$ and delete it from $G$.  
If the resulting graph has a vertex cover of size $k^\ast-1$, then $v$ is in a minimum vertex cover of $G$, if not we return the vertex $v$, and repeat with another vertex.  
Iterating this procedure produces a minimum vertex cover of $G$ in polynomial time.  

While the self-reducibility of many other graph structure search problems are proved using a similar argument, this proof is unusable for IMMORALITIES.  
The analogous argument for IMMORALITIES would require considering all subsets of neighbors of the node $v$ and deleting the corresponding star.  
Since the number of such queries for a given vertex $v$ is only bounded by ${\deg(v)\choose 2}$, this algorithm is not polynomial in time.  
However, this does not prove that IMMORALITIES is not self-reducible, nor does it prove that IMMORALITIES is not NP-complete. 
\end{remark}

\subsection{EXAMPLES OF IMMORALITY NUMBERS}
\label{subsec: computing immorality numbers in some special cases}
As with most NP-hard problems, the problem may become tractable when restricted to special cases.  
We now present a few cases in which star decompositions allow us to compute $\m(G)$ via an application of Theorem~\ref{thm: minimum star decompositions formula}.  
For some graphs, Theorem~\ref{thm: minimum star decompositions formula} makes the computation of $\m(G)$ very simple.  
For instance, the unique minimum star decomposition of a star $S_p$ is itself, and therefore $\m(G)= {p \choose 2}$.  
Similarly, the graph $K_{2,p}$  can be decomposed into two stars, both of which are isomorphic to $K_{1,p}$, in precisely one way, and therefore this is its unique minimum star decomposition.  
It follows from Theorem~\ref{thm: minimum star decompositions formula} that $\m(K_{2,p}) = 2{p\choose 2}$.  
On the other hand, Theorem~\ref{thm: minimum star decompositions formula} certainly has its limitations.  
For instance, let $S_2(p,q)$ denote the gluing of a $p$-star and a $q$-star; i.e. an edge $\{i,j\}$ with $p$ leaves attached to $i$ and $q$ leaves attached to $j$).  
Then, Theorem~\ref{thm: minimum star decompositions formula} implies that $\m(G_2(p,p)) = {p+1\choose 2}+{p\choose 2}$.  
However, if $p\neq q$, then the minimum star decompositions of $G_2(p,q)$ are all size two but need not be isomorphic.  
Therefore, Theorem~\ref{thm: minimum star decompositions formula} does not apply.  

In general, it can be difficult to determine if the minimum star decompositions of a graph are all isomorphic.  
We end this section by computing $\m(G)$ for some slightly more complicated graphs, thereby illustrating the rapidly increasing difficulty level of the problem.  
In subsection~\ref{subsec: the complete bipartite graph} we compute $\m(G)$ for the complete bipartite graph $K_{p,p}$ and in subsection~\ref{subsec: some triangle-free circulants} for some special circulant graphs.

\subsubsection{The complete bipartite graph $K_{p,p}$}
\label{subsec: the complete bipartite graph}
Gillespie and Perlman (2001) note that the maximum number of induced $3$-paths over all skeletons on $p$ nodes, for each $p\leq10$ is given by the complete bipartite graph $K_{\left\lfloor\frac{p}{2}\right\rfloor,\left\lceil\frac{p}{2}\right\rceil}$.  
The number of induced $3$-paths in the graph $K_{\left\lfloor\frac{p}{2}\right\rfloor,\left\lceil\frac{p}{2}\right\rceil}$ is quickly seen to be 
$
a_p = \left\lfloor\frac{p}{2}\right\rfloor\left\lceil\frac{p}{2}\right\rceil\frac{p-2}{2},
$
which is sequence A111384 of (OEIS, 2003).  
Since induced $3$-paths in an undirected graph $G$ are exactly the possible locations of immoralities in a DAG with skeleton $G$, it is reasonable to ask for the immorality number of the complete bipartite graph $K_{p,p}$.  
As one would hope, the immorality number of $K_{p,p}$ turns out to be exactly one half the number of induced $3$-paths.  
We now use Theorem~\ref{thm: minimum star decompositions formula} to compute the immorality number of $K_{p,p}$ via star decompositions.  
To do so, we make one additional observation.
\begin{lemma}
\label{lem: minimum star decompositions of K_p,p}
The minimum star decompositions of $K_{p,p}$ are all isomorphic to 
$$
\{\underbrace{K_{1,p},K_{1,p},\ldots,K_{1,p}}_{p\text{ times}}\}.
$$
\end{lemma}
\begin{proof}
We prove a slightly stronger statement.  
Let $N[v]$ denote the subgraph of a graph $G$ induced by the vertex $v$ and its set of neighbors.  
Let the vertices of $K_{p,p}$ be the partitioned set $A\sqcup B$ where $A := \{a_1,\ldots,a_p\}$ and $B := \{b_1,\ldots,b_p\}$.  
We claim that the minimum star decompositions of $K_{p,p}$ are only 
$
\{N[a_i] : i\in[p]\}
$
and
$
\{N[b_i] : i\in[p]\}.
$
To see this assume otherwise.  
Suppose that $\{S_1,\ldots,S_k\}$ is a minimum star decomposition of $K_{p,p}$, and let $c_i$ denote the center of star $S_i$ for all $i\in[k]$.  
We also set $C := \{c_1,\ldots, c_k\}$.  

Suppose first that $k = p$ and that $A\cap C\neq\emptyset$ and $B\cap C\neq\emptyset$.  
Without loss of generality, assume that $A\cap C = \{a_1,\ldots,a_\ell\}$ for some $\ell<k$.  
Then for all $i>\ell$ it must be that $b_i\in B\cap C$, since otherwise the edge $\{a_i,b_i\}$ would not appear in any star in $\{S_1,\ldots,S_k\}$.  
Since $k=p$, it follows that $B\cap C = \{b_{\ell+1},\ldots,b_p\}$.  
However, this means that for all $i\leq \ell$ and $j\geq \ell+1$, the edges $\{a_j, b_i\}$ are not in any star, which is a contradiction.

Now suppose that $k<p$.  
It is quick to see that if $C\subset A$ or $C\subset B$ then there exist edges of $K_{p,p}$ not contained in stars.  
So $A\cap C\neq\emptyset$.  The proof then follows from applying the same argument as in the case when $k = p$ to derive a contradiction. 
\end{proof}

\begin{theorem}
\label{thm: immorality number of K_p,p}
The immorality number of $K_{p,p}$ is $p{p\choose 2}$.  
\end{theorem}

\begin{proof}
Notice that $K_{p,p}$ is triangle-free.  
By Lemma~\ref{lem: minimum star decompositions of K_p,p}, the minimum star decompositions of $K_{p,p}$ are all isomorphic as forests.  
Moreover, any such minimum star decomposition is induced by a DAG $\mathcal{G}$ on $K_{p,p}$ that has exactly $p$ sinks located along either the node set $A$ or $B$.  
The result then follows from Theorem~\ref{thm: minimum star decompositions formula}.
\end{proof}

\subsubsection{Some triangle-free circulants}
\label{subsec: some triangle-free circulants}

Circulant graphs are natural generalizations of the cycle graphs, and both their independence polynomials and independence numbers have been studied extensively (Brown and Hoshino, 2011; Hoshino, 2008).  
However, there is no general formula for the Fibonacci number, the independence number, nor the independence polynomial of these graphs.  
Similarly, computing $M(G;x)$ or $S(G;x)$ is difficult for a general circulant.  
However, as a corollary to Theorem~\ref{thm: minimum star decompositions formula}, we can compute the immorality number of some triangle-free circulants. 

Recall that a circulant on $p$ nodes is a graph whose nodes are identified with $\ZZ/p\ZZ$, and whose edges are given by a specified \emph{connection set} $C\subset\ZZ/p\ZZ$.  
In the undirected setting, we assume $C$ is closed under additive inverses. 
The circulant on $p$ nodes with connection set $C$ is denoted $X(p,C)$ and has edges $\{i,j\}$ for all pairs $i$ and $j$ satisfying $i-j\in C$.  
We often abbreviate the connection set $C$ via a subset of $\left[\lfloor \frac{p}{2}\rfloor\right]$ by omitting the additive inverse of each element.  

\begin{corollary}
\label{cor: immorality numbers for circulants}
Let $p$ be even, and suppose that $X(p,C)$ is a triangle-free circulant graph containing a $p$-cycle for which the maximum independent subset is of size
$\frac{p}{2}$.  
Then 
$$
\m(X(p,C)) = \frac{p}{2}{2|C| \choose 2}.
$$
\end{corollary}

\begin{proof}
Recall that a set of nodes in a graph $G$ is a minimum vertex cover if and only if its complement is an independent set in $G$.  
Since $X(p,C)$ contains a $p$-cycle, then without loss of generality we can assume that $1\in C$.  
Since $1\in C$ and the maximum independent subset of $X(p,C)$ is equal to the one of $C_p$ of size $p/2$, then any minimum independent set is given by selecting precisely every other vertex of the graph as we walk along the $p$-cycle given by $1\in C$.  
Moreover, such a vertex set is also a minimum vertex cover.  
Thus, if $\{c_1,\ldots, c_{\frac{p}{2}}\}$ is a maximum independent set in $X(p,C)$, then there is only one possible star decomposition with center set $\{c_1,\ldots, c_{\frac{p}{2}}\}$, namely
$$
\{\langle N[c_1]\rangle,\ldots, \langle N[c_{\frac{p}{2}}] \rangle\}\simeq\{K_{1,2|C|},\ldots,K_{1,2|C|}\}.
$$
Since the only two maximum independent sets in $X(p,C)$ share this property, it follows from Lemma~\ref{lem: minimum vertex cover gives star decompositions} that all minimum star decompositions of $G$ are isomorphic.  
Thus, by Theorem~\ref{thm: minimum star decompositions formula} we conclude that
$
\m(X(p,C)) = \frac{p}{2}{2|C| \choose 2}.
$
\end{proof}

Notice that Corollary~\ref{cor: immorality numbers for circulants} applies to any triangle-free circulant with $p$ even, $1 \in C$, which has all other elements of $C$ being odd.  
On the other hand, we cannot apply the same techniques to compute the immorality numbers for $p$ odd, since such circulants may contain nonisomorphic minimum star decompositions.

\section{COMPUTATIONAL ANALYSIS}
\label{sec: computational analysis}

In this section, we describe the computer program we used to test our conjectures and collect relevant statistics.  This program can be found at \url{https://github.com/aradha/mec_generation_tool}, and it expands on the first computer program written for the enumeration of MECs presented (Gillespie and Perlman, 2001).  
For each skeleton on $p\leq 10$ nodes, the Gillespie and Perlman algorithm logged the maximum number of induced $3$-paths, the maximum number of MECs, the total number of MECs, and the size of each class.  
Our program expands on this original program in two ways: for skeletons on $p\leq 10$ nodes, our program collects more data about each skeleton, and it produces all such data for all triangle-free skeletons on $p\leq12$ nodes.    
The new program now catalogues the same information as the original Gillespie and Perlman algorithm for each skeleton as well as the degree sequence of the skeleton, the number of triangles, and the number of immoralities per MEC.  
This additional data, especially in the triangle-free setting, allows us to more carefully analyze how the structure of the skeleton impacts the number and size of its associated MECs.  
In the following, we first provide a brief description of the algorithm and the hashing scheme used.  
We then validate Theorem~\ref{thm: frequency determinism} and discuss the analogous result in the case of unconnected graphs. 

\subsection{THE ALGORITHM}
\label{subsec: the algorithm}
	\begin{figure}[t!]
	\centering
	\includegraphics[width=0.45\textwidth]{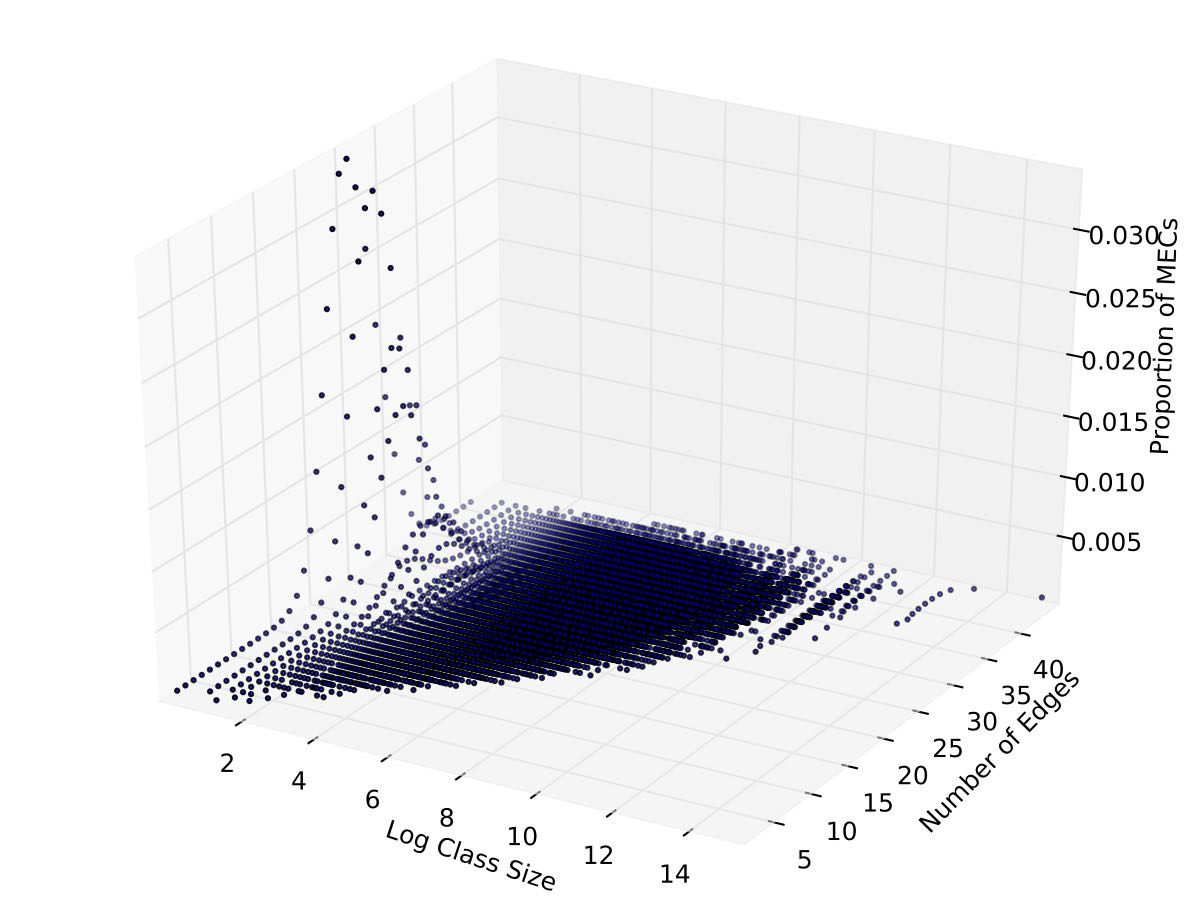}
	\caption{The proportion of MECs on connected graphs with 10 nodes by log class size and number of edges.}
	\label{fig: log-class-size-and-number-of-edges-proportion-of-mecs}
	\end{figure}

There are three main components in our program's data pipeline which we now describe.  
The first component is the main class that reads in skeleton data generated using tools from {\tt nauty} and {\tt Traces} (McKay and Piperno, 2014). The second component is a DAG generator that directly generates all DAGs on a given skeleton.  
Such a generator is realized using the algorithm published by Barbosa and Swarcfiter (1999).  
It is essential to directly generate all DAGs rather than generating all directed graphs and then pruning out the ones containing cycles, since the number of DAGs dominates the number of directed graphs for a large number of vertices.  

The final main component is a DAG enumerator that generates the frequency vector $M(G)_{\mbox{freq}}$ when given the DAGs on a given skeleton $G$.
In order to generate the number of MECs of each size on a given skeleton, this component creates a bit representation for each MEC by first creating a bit mask of the possible immoralities that could occur in the skeleton.  
Each DAG is then traversed.  
If three vertices are found to be in an immorality then the Cantor pairing function is used to hash the triple of their integer labels to the location of the bit in the immorality bit mask.  
Since the Cantor pairing function is invertible and since the number of vertices in each graph is small, we have a valid, non-overflowing hash function.
After comparing the resulting hashes for all DAGs on the given skeleton, a pair of integers is returned for each MEC: the number of immoralities in the MEC and the size of the MEC.  
It is an important feature of the algorithm that this component of the pipeline has access to data on the given skeleton.  
This allows us to collect data on the skeleton in relation to each MEC.  
Using this, for each skeleton we record the number of induced $3$-paths, the degree sequence, the number of edges, and the number of triangles.  
To handle the around 12 million undirected graphs on 10 nodes, we split these graphs into approximately 500 files across 10 directories, allocating 16 threads to process each directory. 
Running this process in parallel takes 5 days as compared to the 253 CPU hours (approximately 94 days) by Gillispie and Perlman.

\subsection{CORRECTNESS OF THE ALGORITHM}
\label{subsec: correctness of the algorithm}
To verify the correctness of our implementation, we matched our program's output with that of the algorithm in (Gillespie and Perlman, 2001).  
In Figure~\ref{fig: log-class-size-and-number-of-edges-proportion-of-mecs} for instance, we can use our program to reproduce the same distribution of the proportion of MECs with respect to class size and number of edges as in (Figure 4, Gillespie and Perlman, 2001).

We also compared performance in terms of speed and memory utilization.  
Our program runs in nearly the exact time measured by Gillespie and Perlman: we measured that our algorithm also takes around three minutes for eight vertices and only a few seconds or milliseconds for fewer vertices.  
We do, however, have better memory utilization than Gillespie and Perlman as the number of bits we store for hashes is dependent on the number of possible immoralities in the skeleton rather than on the number of possible triples of vertices.  
We also use Java to do our data processing.  
Thus, since we only need a print out of the data collected in subsection~\ref{subsec: the algorithm} (3) for each skeleton processed, the garbage collector clears out our hash map allocation after each skeleton.  This allows us to not only log the class size and the number of MECs per skeleton, but also the number of immoralities per class as well as the number of induced $3$-paths, the degree sequence, the number of edges, and the number of triangles.  
Thus, despite the fact that our algorithm only matches the Gillespie and Perlman algorithm in time, it is collecting significantly more data per skeleton.

\subsection{VALIDITY OF THEOREM~\ref{thm: frequency determinism}} 
\label{subsubsec: validity of theorems 2 and 3}
%
\begin{table}
\caption{The 10-node graphs with the same $S(G;x)$ function.  Here, $nK_p$ denotes $n$ disjoint copies of $K_p$.}
\label{table_collisions}
\begin{center}
\begin{tabular}{lll}
Class Size(s) & Skeleton 1 & Skeleton 2\\ \hline
24 & $K_4\sqcup 6K_1$ & $K_3\sqcup 2K_2\sqcup 3K_1$ \\
48 & $K_4\sqcup K_2\sqcup 4K_1$ & $K_3\sqcup 3K_2\sqcup K_1$ \\
144 & $K_4\sqcup K_3\sqcup 3K_1$ & $2K_3\sqcup 2K_2$ \\
720 & $K_6\sqcup 4K_1$ & $K_5\sqcup K_3\sqcup 2K_1$ \\
1440 & $K_6\sqcup K_2\sqcup 2K_1$ & $K_5\sqcup K_3\sqcup K_2$ \\
2880 & $K_6\sqcup 2K_2$ & $K_5\sqcup K_4\sqcup K_1$ \\
72, 24 & $K_4\sqcup I_3\sqcup 3K_1$ & $K_3\sqcup I_3\sqcup 2K_2$ \\
\end{tabular}
\end{center}
\end{table}
After running the algorithm on all connected graphs with up to ten nodes, we verified that there was no pair of skeleta with $p\leq 10$ nodes that have the same frequency vector $M(G)_{\mbox{freq}}$.  
This indicates that the MEC frequency vectors $M(G)_{\mbox{freq}}$ (or equivalently the arithmetic generating functions $S(G;x)$) bijectively map to skeletons of connected graphs up to ten nodes.  
Similarly, when we ran our algorithm on all graphs with ten nodes including graphs that were not necessarily connected, we found that the only collisions occurred on graphs $G$ and $H$ with the following property: Let $G=G_1\sqcup\cdots \sqcup G_m$  and $H=H_1\sqcup\cdots \sqcup H_n$ be the decompositions of $G$ and $H$ into connected components. Let $G\cap H$ denote the set consisting of the connected components that are shared between $G$ and $H$ up to isomorphism. Now let $G\setminus G\cap H = G_{i_1}\sqcup\cdots \sqcup G_{i_m}$ and $H\setminus G\cap H = H_{j_1}\sqcup\cdots \sqcup H_{j_n}$, where $i_1,\dots , i_m\in [m]$ and $j_1,\dots , j_n\in [n]$, be the remaining subgraphs.  Then 
$
\prod_{k=1}^m |G_{i_k}| = \prod_{\ell=1}^n |H_{j_{\ell}}|.
$
Over all graphs with ten nodes, there are seven such examples that occurred. These are shown in Table~\ref{table_collisions}.

\section{Discussion}
\label{sec: discussion}

Understanding the number and size of MECs is important since it tells us about the complexity and uncertainty of DAG model recovery.  
In this paper, we introduced a pair of generating functions, $M(G;x)$ and $S(G;x)$, that count the number of MECs on a skeleton $G$ by their number of immoralities and size, respectively.  
This constitutes a novel approach to the MEC enumeration problem that yields connections to classically studied problems in combinatorial optimization.  
In particular, we observed that computing the degree of $M(G;x)$ for triangle-free graphs relates to the vertex-cover problem for $G$ and its associated star decompositions.  
These connections allowed us to prove that computing the degree of $M(G;x)$ is NP-hard, thus demonstrating that counting the number of MECs on $G$ must be hard as well.  
Alternatively, we observed the complexity of enumerating MECs by size by showing that $S(G;x)$ is distinct for every connected graph $G$ on $p\leq 10$ nodes.  
The connections to classical problems revealed here suggest that the number of MECs for sparse graphs can be better understood by a closer examination of $M(G;x)$ and $S(G;x)$.  
In particular, it is natural to ask how $M(G;x)$ relates to the enumeration of vertex covers (or related graphical structures) for fixed families of sparse graphs.  
Future work is required to address questions of this nature.

\subsubsection*{Acknowledgements}
We wish to thank Brendan McKay for helpful advice in the use of the programs {\tt nauty} and {\tt Traces}.  
Liam Solus was supported by an NSF Mathematical Sciences Postdoctoral Research Fellowship (DMS - 1606407).  Caroline Uhler was partially supported by DARPA (W911NF-16-1-0551), NSF (1651995) and ONR (N00014-17-1-2147).

\subsubsection*{References}
S.~A.~Andersson, D.~Madigan, and M.~D.~Perlman. \emph{A characterization of Markov equivalence classes for acyclic digraphs}. Annals of Stat. 25.2 (1997): 505--541.

V.~C.~Barbosa and J.~L.~Szwarcfiter. \emph{Generating all the acyclic orientations of an undirected graph}. Information Processing Letters 72.1 (1999): 71--74.




J.~Brown and R.~Hoshino. \emph{Well-covered circulant graphs}. Discrete Mathematics 311.4 (2011): 244--251.

P.~Cain. \emph{Decomposition of complete graphs into stars}. Bulletin of the Australian Mathematical Society 10.01 (1974): 23--30.


D.~M.~Chickering. \emph{Learning equivalence classes of Bayesian-network structures}. Journal of Machine Learning Research 2 (2002): 445--498.

E.~Cohen and M.~Tarsi. \emph{NP-completeness of graph decomposition problems}. Journal of Complexity 7.2 (1991): 200--212.

J.~A.~Ellis-Monaghan and C.~Merino. \emph{Graph polynomials and their applications I: The Tutte polynomial}. Structural Analysis of Complex Networks. Birkh\"{a}user Boston, 2011. 219--255.

	
N.~Friedman, M.~Linial, I.~Nachman and D.~Peter. \emph{Using Bayesian networks to analyze expression data}. Journal of Computational Biology 7 (2000): 601--620.

M.~R.~Garey and D.~S.~Johnson. \emph{Computers and intractability: a guide to the theory of NP-completeness}. A Series of Books in the Mathematical Sciences. WH Freeman and Company, New York, NY 25.27 (1979): 141.

S.~B.~Gillispie. \emph{Formulas for counting acyclic digraph Markov equivalence classes.} Journal of Statistical Planning and Inference 136.4 (2006): 1410-1432.

S. B.~Gillispie and M. D.~Perlman. \emph{Enumerating Markov equivalence classes of acyclic digraph models.} Proc. of the Seventeenth Conference on Uncertainty in Artificial Intelligence. Morgan Kaufmann Publishers Inc., 2001.


Y.~He and B.~Yu. \emph{Formulas for counting the sizes of Markov equivalence classes of directed acyclic graphs.} Preprint available on the arXiv: \url{https://arxiv.org/pdf/1610.07921.pdf} (2016).  

R.~Hoshino. \emph{Independence polynomials of circulant graphs}. Library and Archives Canada, 2008.

R.~M.~Karp \emph{Reducibility among combinatorial problems}. Complexity of Computer Computations. Springer US (1972): 85--103.

V.~E.~Levit and E.~Mandrescu. \emph{The independence polynomial of a graph -- a survey}. Proceedings of the 1st International Conference on Algebraic Informatics. Vol. 233254. 2005.

C.~Lin, and T-W.~Shyu. \emph{A necessary and sufficient condition for the star decomposition of complete graphs}. Journal of Graph Theory 23.4 (1996): 361--364.

J.~H.~van Lint and R.~M.~Wilson. \emph{A Course in Combinatorics}. Cambridge University Press, 2001.

B.~D.~McKay and A.~Piperno. \emph{Practical graph isomorphism, II}. J. of Symbolic Comp. 60 (2014): 94--112.

C.~Meek. \emph{Causal inference and causal explanation with background knowledge}. Proc. of the Eleventh Conference on Uncertainty in Artificial Intelligence (1995): 403--410.

J.~Pearl. \emph{Causality: Models, Reasoning, and Inference}. Cambridge University Press, Cambridge, 2000.

S.~Poljak. \emph{A note on stable sets and colorings of graphs.} Commentationes Mathematicae Universitatis Carolinae 15.2 (1974): 307--309.

H.~Prodinger and R.~F.~Tichy. \emph{Fibonacci numbers of graphs}. Fibonacci Quarterly 20.1 (1982): 16--21.


J.~M.~Robins, M.~A.~Hern\'an and B.~Brumback. \emph{Marginal structural models and causal inference in epidemiology}. Epidemiology 11.5 (2000): 550--560.

N.~J.~Sloane. \emph{The On-Line Encyclopedia of Integer Sequences}. (2003).


P.~Spirtes, C.~N.~Glymour and R.~Scheines. \emph{Causation, Prediction, and Search}. MIT Press, Cambridge, 2001.

B.~Steinsky. \emph{Enumeration of labelled chain graphs and labelled essential directed acyclic graphs}. Discrete Mathematics 270.1 (2003): 267-278.



K.~Ushio, S.~Tazawa and S.~Yamamoto. \emph{On claw-decomposition of a complete multipartite graph}. Hiroshima Mathematical Journal 8.1 (1978): 207--210.

T.~Verma and J.~Pearl. \emph{An algorithm for deciding if a set of observed independencies has a causal explanation}. Proc. of the Eighth International Conference on Uncertainty in Artificial Intelligence. Morgan Kaufmann Publishers Inc., 1992.
 
S.~Wagner. \emph{Asymptotic enumeration of extensional acyclic digraphs}. Algorithmica 66.4 (2013): 829--847.

\end{document}